\documentclass[10pt]{article}

\usepackage{verbatim}        

\usepackage{color}         
            
\usepackage{amsmath}
\usepackage{amsthm} 
\usepackage{amssymb}
\usepackage{graphicx}
\usepackage{hyperref}

\usepackage{times}
\usepackage{enumerate}

\def\titlerunning#1{\gdef\titrun{#1}}
\makeatletter
\def\author#1{\gdef\autrun{\def\and{\unskip, }#1}\gdef\@author{#1}}
\def\address#1{{\def\and{\\\hspace*{18pt}}\renewcommand{\thefootnote}{}%
\footnote {#1}}%
\markboth{\autrun}{\titrun}}
\makeatother
\def\email#1{e-mail: #1}
\def\subjclass#1{{\renewcommand{\thefootnote}{}%
\footnote{\emph{Mathematics Subject Classification (2010):} #1}}}
\def\keywords#1{\par\medskip
\noindent\textbf{Keywords.} #1}

\theoremstyle{plain}

\newtheorem{thm}{Theorem}
\theoremstyle{plain}
\newtheorem{lem}[thm]{Lemma}
\newtheorem{corol}[thm]{Corollary}

\theoremstyle{definition}
\newtheorem{defn}[thm]{Definition}
\newtheorem{openproblem}{Open Problem}

\theoremstyle{definition}
\newtheorem{ex}[thm]{Example}

\newcommand{\rank}[1]{\mathrm{rank}(#1)}

\newcommand{\maxdegree}[1]{\mathrm{\mathcal{MD}}(#1)}

\newcommand{\zrms}[4]{\mathcal{M}^0[#1;#2,#3;#4]}

\newcommand{\lb}{\langle}
\newcommand{\rb}{\rangle}

\newcommand{\gr}{\mathcal{R}}
\newcommand{\gl}{\mathcal{L}}
\newcommand{\gh}{\mathcal{H}}
\newcommand{\gj}{\mathcal{J}}

\newcommand{\gd}{\mathcal{D}}

\numberwithin{equation}{section}

\usepackage{epic,eepic,ecltree} % for drawing trees

\input xy
\xyoption{all}

\usepackage{xy} % for drawing pictures
\newcommand{\edge}[1]{\ar@{-}[#1]}
\newcommand{\lulab}[1]{\ar@{}[l]_<<{#1}}
\newcommand{\rulab}[1]{\ar@{}[r]^<<{#1}}
\newcommand{\ldlab}[1]{\ar@{}[l]^<<{#1}}
\newcommand{\rdlab}[1]{\ar@{}[r]_<<{#1}}
\newcommand{\node}{*+[o][F-]{ }}

\usepackage{lscape}
\usepackage{rotating}

\usepackage{geometry}

\newcounter{case}
\newenvironment{case}
{\stepcounter{case} \noindent \textbf{Case~\arabic{case}:}\itshape}{}

\newcounter{subcase}[case]

\newcounter{subsubcase}[subcase]

\DeclareMathOperator{\im}{im}

\begin{document}

\titlerunning{Generating finite semigroups}
\title{The minimal number of generators of a finite semigroup}
\author{Robert D. Gray}
\date{}

\maketitle

\address{
R. D. Gray: School of Mathematics, 
University of East Anglia, UK; 
\email{Robert.D.Gray@uea.ac.uk}
}


\vspace{-12mm}

\begin{abstract}
The rank of a finite semigroup is the smallest number of elements required to generate the semigroup. A formula is given for the rank of an arbitrary (non necessarily regular) Rees matrix semigroup over a group. The formula is expressed in terms of the dimensions of the structure matrix, and the relative rank of a certain subset of the structure group obtained from  subgroups generated by entries in the structure matrix, which is assumed to be in Graham normal form. This formula is then applied to answer questions about minimal generating sets of certain natural families of transformation semigroups. In particular, the problem of determining the maximum rank of a subsemigroup of the full transformation monoid (and of the symmetric inverse semigroup) is considered.  

\end{abstract}

\section{Introduction and preliminaries}

J. M. Howie was very interested in certain problems that lie at the boundary between semigroup theory and combinatorics. In particular, he wrote numerous papers concerned with the following natural problem: given a finite semigroup $S$ what is the smallest number of elements needed to generate $S$? In semigroup theory this number is usually referred to as the \emph{rank} of the semigroup, so
\[
\rank{S} = \min \{ |A| : A \subseteq S \ \mbox{and} \ \lb A \rb = S  \}, 
\]
where $\lb A \rb$ denotes the subsemigroup generated by the set $A$. (Note that in group theory this number is usually denoted $d(G)$.) 
Howie, along with various co-authors, wrote a number of influential papers on ranks of semigroups; see for instance 
\cite{key-2,key-12,key-13,key-15}.
Since then, many more papers on ranks of semigroups have been written. 
Some recent examples where ranks of semigroups have been considered include 
\cite{Zhao2011,Fernandes2011, East2011, East2010, Araujo2009, Fernandes2009},
among others.  
Some results about semigroups with the property that all minimal generating sets have the same cardinality are proved in \cite{Doyen84, Doyen91}. 

The vast majority of papers in this area are concerned with finding minimal cardinality generating sets for certain naturally arising semigroups; most often semigroups of transformations, matrix semigroups, and more generally endomorphism monoids of various natural combinatorial or algebraic structures. For some of these examples one may observe that the methods for determining the rank in one class of examples are not so different from those used for another class. For instance, the steps required for computing the ranks of the proper two-sided ideals of the full transformation monoid $T_n$, the symmetric inverse semigroup $I_n$, and the partial transformation monoid $P_n$ all follow more or less the same pattern; see \cite{MR92c:20123,key-2,key-13}. Thus it is natural to ask whether some general theory of ranks of finite semigroups might be developed which could then be applied to 
examples like these.

This is the central theme of this paper: to give some abstract theory regarding ranks of finite semigroups and (perhaps more importantly) to explain how this theory can be usefully applied to compute the ranks of various concrete examples of transformation semigroups. Part of the motivation for the paper is to try and bring to the attention 
of researchers working in this area
some general tools that can be useful for 
finding minimal generating sets 
of transformation semigroups
in certain situations.

The general idea is as follows. Let $S$ be a finite semigroup. Recall that the principal factors of $S$ (which may be thought of as the basic building blocks of $S$) are obtained by taking a $\gj$-class $J$ and forming a semigroup $J^* = J \cup \{ 0 \}$ with multiplication given by
\[
s \cdot t =
\begin{cases}
st & \mbox{if $s, t, st \in J$} \\
0 & \mbox{otherwise},
\end{cases}
\]
and, of course, $0s=s0=0$ for all $s \in J^*$. If $S$ is finite then every principal factor $J^*$ is either a semigroup with zero multiplication (when $J$ contains no idempotents) or $J^*$ is a completely $0$-simple semigroup and hence, by the Rees theorem, is isomorphic to some regular Rees matrix semigroup $\zrms{G}{I}{\Lambda}{P}$ over a group $G$ with structure matrix $P$.
(For background on basic concepts from semigroup theory such as Green's relations and Rees matrix semigroups we refer the reader to \cite{key-51}.) 

Let $J_1, \ldots, J_m$ be the maximal $\gj$-classes of $S$. Now if $A \subseteq S$ generates $S$ then for each $i$, every element of $J_i$ must be expressible as a product of elements from $A \cap J_i$. In other words, $A \cap J_i$ is a generating set for the principal factor $J_i^*$. This immediately implies that
\[
\rank{S} \geq \sum_{i=1}^m \rank{J_i^*}, \tag{$\dagger$}
\]
where $J_1, \ldots, J_m$ are the maximal $\gj$-classes of $S$. 
When $J_i^*$ is a semigroup with zero multiplication then clearly $\rank{J_i^*} = |J_i|$. On the other hand, when $J_i$ contains an idempotent, $J_i^*$ will be isomorphic to some regular Rees matrix semigroup $\zrms{G}{I}{\Lambda}{P}$ over a group $G$ with structure matrix $P$. This explains why the starting point for the development of any general theory of minimal generating sets of finite semigroups must be to try and understand minimal generating sets of Rees matrix semigroups. 
Of course in general ($\dagger$) will not be an equality, but in many natural examples it is. 
For instance, 
the examples mentioned above, namely the proper ideals of $T_n$, $I_n$ and $P_n$ all fall into this category. Also, even when ($\dagger$) is not an equality the question
of determining the ranks of the maximal $J_i^*$ remains relevant since 
 any minimal generating set for $S$ must contain minimal generating sets for each of the principal factors of the maximal $\gj$-classes.

The paper is structured as follows. In \S2, generalising slightly the main result of \cite{Gray2005}, a formula will be presented which gives the rank of an arbitrary (not necessarily regular) Rees matrix semigroup over a group, given in terms of the underlying group $G$ and the structure matrix $P$. 
In order to understand this formula we shall first need to go through the basics of the theory of Graham normalization for Rees matrix semigroups. 
We shall then go on to see how this formula can be usefully applied to investigate questions about minimal generating sets of certain semigroups of transformations.  In particular in \S3 we shall make some preliminary investigations into the following natural problem: given a transformation semigroup $S$ on $n$ points, how many generators (in the worst case) are needed to generate $S$? In other words, what is
$
\max\{ \rank{S} : S \leq T_n \}?
$
This is a classical question for the symmetric group $S_n$ in the theory of finite permutation groups where it is known that every group $G \leq S_n$ ($n>3$) is generated by at most $\lfloor n / 2 \rfloor$ elements; see \cite{McIverNeumann1987}. 

Throughout, since we shall mainly be considering semigroups that have a zero element,  we will always include the zero in any given subsemigroup. As a consequence of this, by $\lb X \rb$ we will mean all the elements that can be written as products of elements of $X$, plus zero if necessary. We use $E(S)$ to denote the set of idempotents in a semigroup $S$.

\section{Graham--Houghton graphs and Graham normal form}

As part of the more advanced topics covered in the recent monograph on finite semigroup theory by Rhodes and Steinberg  is a section devoted to R. Graham's description of the idempotent generated subsemigroup of a $0$-simple semigroup; see
\cite[Section~4.13]{SteinbergBook2009}. Their motivation for including this material comes from the fact that a detailed study of the idempotent generated subsemigroup of a finite semigroup 
can be important for understanding complexity. 
Graham's theorem describes the idempotent generated subsemigroup of a Rees matrix semigroup. These results were later rediscovered by Houghton who gave them a topological interpretation \cite{Houghton1977}. Graham--Houghton graphs, and $2$-complexes, have also recently arisen as an important tool in the study of, so-called, free idempotent generated semigroups; see for example
\cite{Brittenham2009, DolinkaGray, Easdown2010, Gray2012(1), Gray2012(2)}.

Something which is possibly less commonly known is the importance of Graham's ideas when one is interested in finding small generating sets for finite semigroups. The connection comes from the fact that in many natural examples minimal generating sets may be found by taking a disjoint union of minimal generating sets for a collection $J_i^* (i \in I)$ of principal factors of $S$, and so the problem comes down to finding minimal generating sets for the corresponding Rees matrix semigroups. This, in turn, may be reduced to questions about the maximal subgroups when  developing Graham's ideas in the appropriate way. 

As in \cite{SteinbergBook2009} we find here that for applications 
we shall benefit from considering 
arbitrary Rees matrix semigroups over groups, and not just regular ones. 
So throughout by a Rees matrix semigroup over a group we shall mean an arbitrary (non-necessarily regular) Rees matrix semigroup.

\subsection{Graham--Houghton graphs and Graham's Theorem} 
Throughout we shall, for the most part, follow the notation and conventions of \cite[Section~4.13]{SteinbergBook2009}.
A \emph{graph} (in the sense of Serre \cite{SerreTrees}) consists of a set $V$ of vertices, a set $E$ of edges and three functions
\[
\iota: E \rightarrow V, \quad
\tau: E \rightarrow V, \quad
{}^{\overline{\ \ }}:E \rightarrow E,
\]
called initial, terminal and inverse (respectively),
where $e \mapsto \overline{e}$ is a fixed point free involution satisfying $\iota e = \tau \overline{e}$ and $\tau e = \iota \overline{e}$.  
A \emph{non-empty path} $p$ will consist of a sequence of edges $e_1 e_2 \ldots e_n$ with $\tau e_i = \iota e_{i+1}$ $(1 \leq i < n)$. The initial, terminal and involution functions extend naturally to paths by defining 
\[
\iota p = \iota e_1, \quad
\tau p = \tau e_n, \; \mbox{and} \; \;
\overline{p} = \overline{e_n} \ldots  \overline{e_2} \; \overline{e_1}.
\]
We admit an empty path $1_v$ at each vertex $v$.
If $p$ and $q$ are paths with $\tau p = \iota q$ then we can form the \emph{product path} $p q$ consisting of the edges of $p$ followed by the edges of $q$. 
A graph is \emph{connected} if any two vertices can be connected by a path, and the connected components are the maximal connected subgraphs. Pairs $\{ e, \overline{e} \}$ are called \emph{geometric edges}, and an \emph{orientation} of $\Gamma$ is given by choosing a representative from each geometric edge. The chosen edge is said to be \emph{positively oriented}. A bipartite graph admits two natural orientations, namely orienting geometric edges to point at one part or the other of the bipartition. 
\begin{defn}[$G$-labelled graph] Let $\Gamma$ be a graph and $G$ be a group. A \emph{$G$-labelling} of $\Gamma$ is a map $l: E \rightarrow G$ such that $l \overline{e} = (l e)^{-1}$. We call $(\Gamma,l)$ a $G$-labelled graph.
\end{defn}
Of course, given a $G$-labelled graph $(\Gamma,l)$, the labelling map $l$ extends in an obvious way to paths by setting 
$
l(e_1 \ldots e_n) = (l e_1)(l e_2) \ldots (l e_n), 
$
and we clearly have $(l \overline{p}) = (l p)^{-1}$. There is a homotopy theory for $G$-labelled graphs that lies in the background of some of the results presented here. We shall not go into the details of this here, the interested reader should consult \cite[Section~4.13]{SteinbergBook2009}.

\begin{defn}[Graham--Houghton graph of a Rees matrix semigroup]\label{def_GH}
Let $S = \zrms{G}{I}{\Lambda}{P}$ be a Rees matrix semigroup over a group $G$ with structure matrix $P = (p_{\lambda i})$, where we assume $I \cap \Lambda$ is empty. The Graham--Houghton graph of $S$ (also called the incidence graph of $S$) denoted $\Gamma(S)$, has vertex set $V = I \cup \Lambda$ and edge set 
\[
E = \{ 
(i, \lambda), (\lambda, i) : i \in I, \; \lambda \in \Lambda, \; p_{\lambda i} \neq 0 
\}.
\]
The involution is given by $\overline{(x,y)} = (y,x)$, and $\iota(x,y) = x$ and $\tau(x,y)=y$. 

The structure matrix $P$ gives $\Gamma(S)$ the structure of a $G$-labelled graph $(G,l_P)$ by defining $l_P(i,\lambda) = p_{\lambda i}^{-1}$ and $l_P(\lambda,i) = p_{\lambda i}$ for $i \in I, \lambda \in \Lambda$. We orient $\Gamma(S)$ by taking edges in $\Lambda \times I$ as the set of positively oriented edges. 
\end{defn}

Throughout we shall use $\Gamma(S)$ to denote the Graham--Houghton graph (without labels), and  $(\Gamma(S),l_P)$ to denote the corresponding $G$-labelled graph. 
Graham's fundamental observation was that the oriented $G$-labelled graph $(\Gamma(S),l_P)$ encodes the idempotent generated subsemigroup of $S = \zrms{G}{I}{\Lambda}{P}$. Before presenting his result, we first make a few basic observations about the unlabelled graph $\Gamma(S)$. 

\

\noindent (i) There is a natural bijection between the geometric edges of $\Gamma(S)$ and the non-zero idempotents in $S$ given by $(i,\lambda) \leftrightarrow (i, p_{\lambda i}^{-1}, \lambda)$. 

\vspace{2mm}

\noindent (ii) When $S$ is $0$-simple (i.e. when the matrix $P$ is regular) the isomorphism type of the graph $\Gamma(S)$ only depends on the isomorphism type of $S$ and not on the choice of Rees matrix representation. This is because in this situation the graph $\Gamma(S)$ simply records those $\gh$-classes in the non-zero $\gd$-class of $S$ that contain idempotents (that is, the group $\gh$-classes). 

\vspace{2mm}

\noindent (iii) $S$ being $0$-simple is equivalent to saying that $\Gamma(S)$ has no isolated vertices. 

\

Let $\mathcal{P}_{i,\lambda}$ denote the set of all paths in $\Gamma(S)$ from $i$ to $\lambda$. Then $(\Gamma(S),l_P)$ can be used in a direct way to describe the idempotent generated part of $S$ as the following  result shows. 
\begin{thm}\label{thm_GrahamHowie}
Let $S = \zrms{G}{I}{\Lambda}{P}$ be a Rees matrix semigroup. Then
\[
\lb E(S) \rb =
\{
(i, l_P(\pi), \lambda) : i \ \mbox{and} \ \lambda \ \mbox{belong to the same connected component of $\Gamma(S)$ and $\pi \in \mathcal{P}_{i,\lambda}$}
\} \cup \{ 0 \}. 
\]
\end{thm}
This result follows from the straightforward correspondence between non-zero products of idempotents of $S$ and values of labels of paths from $I$ to $\Lambda$ in $(\Gamma(S),l_P)$; for a detailed proof see, for instance, 
\cite{key-4}. 

\subsection{Isomorphism theorem and Graham normal form}

Now, at this stage it is not clear how far Theorem~\ref{thm_GrahamHowie} takes us, since computing all values of all possible paths in $\Gamma(S)$ would seem just as involved as computing by hand what the idempotents generate in the first place. Also, it is not yet clear what any of this has to do with finding small generating sets for Rees matrix semigroups. 

The answer to the first of these questions is given by a result of Graham which we now describe. The starting point is to recall that two Rees matrix semigroups $\zrms{G}{I}{\Lambda}{P}$ and $\zrms{G}{I}{\Lambda}{Q}$ over the same group $G$ and with the same index sets $I$ and $\Lambda$ may be isomorphic even if $P \neq Q$. Of course, basic operations like permuting rows and columns of $P$ will not change the isomorphism type, but more interesting transformations of $P$ can be carried out. The process of changing $P$ while leaving the isomorphism type of $S$ unchanged is known as \emph{normalization}. 
For instance, for regular Rees matrix semigroups we have the following well-known result. 
\begin{thm}\cite[Theorem~3.4.1]{key-51}\label{isomorphismtheorem}
Two regular Rees matrix semigroups $S = \zrms{G}{I}{\Lambda}{P}$ and $T =
\zrms{K}{J}{M}{Q}$ are isomorphic if and only if there exist an isomorphism $\theta: G
\rightarrow K$, bijections $\psi:I \rightarrow J$, $\chi:\Lambda \rightarrow M$ and
elements $u_i \ (i \in I), \nu_{\lambda} \ (\lambda \in \Lambda)$ such that
$
p_{\lambda i} \theta = \nu_{\lambda} q_{\lambda \chi, i \psi} u_i
$
for all $i \in I$ and $\lambda \in \Lambda$.
\end{thm}

Roughly speaking Graham's result shows that for an arbitrary Rees matrix semigroup, the structure matrix
$P$ may be normalized in such a was that $\lb E(S) \rb$ can be computed just using information about the non-zero elements in $P$, and the subgroups of $G$ that they generate. This is a huge step forward, since it means $\lb E(S) \rb$ can be determined without having to consider all values of all possible paths in $(\Gamma(S),l_P)$. In more detail, if $S$ is a Rees matrix semigroup then we can normalize the matrix in
a special way using the graph $\Gamma(S)$.

\begin{thm}{\cite[Theorem~4.13.11]{SteinbergBook2009}}
\label{identityspanningtree}
Let $S = \zrms{G}{I}{\Lambda}{P}$ be a Rees matrix semigroup and let $\Gamma(S)$ be its Graham--Houghton graph. For any spanning forest $\mathcal{F}$ of $\Gamma(S)$ it is possible to normalize $S$ to obtain $U = \zrms{G}{I}{\Lambda}{Q} \cong \zrms{G}{I}{\Lambda}{P}$ such that $\Gamma(U) \cong \Gamma(S)$ (via the identity map on $I \cup \Lambda$) and $l_Q(e) = 1_G$ (the identity element of $G$) for every edge $e$ in the spanning forest $\mathcal{F}$. 
\end{thm}
Graham \cite{key-29} was the first to realise that the structure matrix of a Rees matrix semigroup may be
normalized in this special way. Once this process has been carried out we say that the matrix has been put into \emph{Graham normal form}. Graham realised that once the structure matrix has been normalized in this way, a very nice description of the idempotent generated subsemigroup of $S$ may then be given. 

\begin{thm}\cite[Theorem~2]{key-29}\label{GNF} 
Let $S = \zrms{G}{I}{\Lambda}{P}$ be a
  Rees matrix semigroup, and let $\Gamma(S)$ be the Graham--Houghton graph of $S$. 
Let $I'$ and $\Lambda'$ be the respective sets of isolated vertices of $\Gamma(S)$ from $I$ and $\Lambda$. 
Let $(I_1 \cup \Lambda_1), \ldots, (I_n \cup \Lambda_n)$ be the connected components of 
$(I \setminus I') \cup (\Lambda \setminus \Lambda')$.   
Then there is a $\Lambda' \times I'$ zero matrix $C_N$ and a regular Rees matrix 
$C_R: (\Lambda \setminus \Lambda') \times (I \setminus I') \rightarrow G^0$ such that:
\begin{enumerate}
\item $S \cong \zrms{G}{I}{\Lambda}{C_R \oplus C_N}$ where 
$
C_R \oplus C_N = 
\begin{pmatrix}
C_R & 0 \\
O & C_N
\end{pmatrix};
$
\item The matrix $C_R$ is block diagonal of the form 
\[
C_R = 
\begin{pmatrix}
C_1 	& 		0 		& 		\ldots 		& 		0 			\\
0		&		C_2	&		\ldots		&		\vdots   	\\
\vdots 	&		\vdots	&		\ddots		&		\vdots		\\
0		&		\ldots	&		0			&		C_n
\end{pmatrix}
\]
where $C_i : \Lambda_i \times I_i \rightarrow G^0$ is a connected regular Rees matrix semigroup over $G^0$ for $i=1,\ldots,n$;
\item $\langle E(S) \rangle = \bigcup_{i=1}^n{ \zrms{G_i}{I_i}{\Lambda_i}{C_i}}$, where $G_i$ is the subgroup generated by the non-zero entries of $C_i$. 
\end{enumerate}
\end{thm}
In (2) saying that $C_i : \Lambda_i \times I_i \rightarrow G^0$ is a connected regular Rees matrix semigroup over $G^0$ just means that the corresponding Graham--Houghton graph is connected. So the blocks $C_1, \ldots, C_n$ correspond to the connected components of the graph $\Gamma(S)$. The union of Rees matrix semigroups given in part (3)  denotes the zero sum of the semigroups. 
Given a finite semigroup $S$, it is known that a Graham normalization for $S$ can be computed from the multiplication table of $S$ in polynomial time. Graham normalization has shown to be an effective tool in the study of finite semigroups, for instance 
the regular Type II elements of any semigroup were described in coordinates in \cite{Rhodes1972} making use of Graham normalizations.

We note that in \cite{key-29} Graham does substantially more than just use his graph theoretic approach to describe the idempotent generated subsemigroup of a finite $0$-simple semigroup. He actually works with a more general graph than the one given in Definition~\ref{def_GH} above, which is constructed from a finite $0$-simple semigroup together with a given fixed subset  $X$ of $S$. When one is interested in the idempotent generated subsemigroup of $S$ then attention may be restricted to the case $X = E(S)$, resulting in the graph given in Definition~\ref{def_GH}; see \cite[Section~4]{key-29}. 
In addition to the results above, in \cite{key-29} Graham also uses this general graph-theoretic approach to describe: the maximal nilpotent subsemigroups and maximal subsemigroups of finite $0$-simple semigroups. The maximal completely simple subsemigroups and maximal zero subsemigroups can also be described using his graph. This approach can be useful for investigating other notions of rank such as nilpotent and idempotent rank (see \cite{Gray2008} for example).      
In another related, and largely forgotten beautiful paper \cite{GGR68} Graham, Graham and Rhodes use this approach to characterise maximal subsemigroups of finite semigroups in general.

\subsection{The rank of a Rees matrix semigroup}

Using Graham's result we can now give a formula for the rank of an arbitrary Rees matrix semigroup. This result was proved for finite regular Rees matrix semigroups in 
\cite[Corollary~9.1]{Gray2005}, and extending to arbitrary finite Rees matrix semigroups is straightforward; we give the details of this extension in the proof below. 
Before stating the result we shall need one more concept. Given a subset $A$ of a semigroup $S$, we define the \emph{relative rank of $S$ modulo $A$} as the minimal number of elements of $S$ that need to be added to $A$ in order to generate the whole of $S$, that is:
\[
\rank{S:A} = \min\{|X| : \lb A \cup X \rb = S \}. 
\]
\begin{thm}\label{rankreesnonregular}
Let $S=\zrms{G}{I}{\Lambda}{P}$ be a Rees matrix semigroup, with structure matrix in Graham normal form 
\[
P = 
\begin{pmatrix}
C_R & 0 \\
0 & C_N
\end{pmatrix}
\]
where $C_N$ is a $\Lambda' \times I'$ zero matrix and $C_R: (\Lambda \setminus \Lambda') \times (I \setminus I') \rightarrow G^0$ is a block diagonal matrix of the form
\[
C_R = 
\begin{pmatrix}
C_1 	& 		0 		& 		\ldots 		& 		0 			\\
0		&		C_2	&		\ldots		&		\vdots   	\\
\vdots 	&		\vdots	&		\ddots		&		\vdots		\\
0		&		\ldots	&		0			&		C_n
\end{pmatrix}
\]
where $C_i : \Lambda_i \times I_i \rightarrow G^0$ is a connected regular Rees matrix semigroup over $G^0$ for $i=1,\ldots,n$. If $P$ is a zero matrix then $\rank{S} = |G||I||\Lambda|$, otherwise:
\[
\rank{S} = 
\max(|I \setminus I'|, |\Lambda \setminus \Lambda'|, \sigma_{\mathrm{min}} + n - 1) + |I'| + |\Lambda'|,
\]
where, with $H_i$ denoting the subgroup of $G$ generated by the non-zero entries in $C_i$, we define
\[
\sigma_{\mathrm{min}} = \min\{ 
\rank{
G:\bigcup_{i=1}^n g_i H_i g_i^{-1} \; | \; g_1, \ldots, g_n \in G
}
\}.
\]
\end{thm}
\begin{proof} 
Clearly if $P$ is a zero matrix then $S$ is a semigroup with zero multiplication 
and $\rank{S} = |G||I||\Lambda|$, so suppose otherwise, that is, suppose that $C_R$ is non-empty. 

Let $T = \zrms{G}{I \setminus I'}{\Lambda \setminus \Lambda'}{C_R}$ which is a regular Rees matrix semigroup over $G$ with structure matrix $C_R$ in Graham normal form. Now by \cite[Corollary~9.1]{Gray2005} it follows that
\[
\rank{T} = \max(|I\setminus I'|, |\Lambda \setminus \Lambda'|, \sigma_{\mathrm{min}} + n - 1)
\]
with $\sigma_{\mathrm{min}}$ defined as in the statement of the theorem. Therefore, to complete the proof of the theorem it will suffice to show
\begin{equation}\label{eqn_rank}
\rank{S}  =  \rank{T} + |I'| + |\Lambda'|. 
\end{equation}
Before proving this equality we shall first need some basic observations about $S$. Define
\[
U = (I \times G \times \Lambda') \cup \{ 0 \}, 
\quad \mbox{and} \quad
V = (I' \times G \times \Lambda) \cup \{ 0 \}. 
\]
Straightforward computations show that both $U$ and $V$ are two-sided ideals of $S$, and hence so are $U \cap V$ and $U \cup V$. 
Also, \[\lb S \setminus (U \cap V) \rb = S. \tag{$\dagger$}\] To see this, fix $\lambda_0 \in \Lambda \setminus \Lambda'$ and $i_0 \in I \setminus I'$ such that $p_{\lambda_0 i_0} \neq 0$ (this is possible since $C_R$ is regular). Then for all $i' \in I'$ and $\lambda' \in \Lambda'$ we have
\[
\{ (i',g,\lambda') : g \in G \} = 
\{ (i',g,\lambda_0) : g \in G \}
\{ (i_0,g,\lambda') : g \in G \},
\]
which establishes the claim. Now we return to the proof of the equality \eqref{eqn_rank}. 

\vspace{2mm}

\noindent ($\leq$) We begin by showing that a generating set with the desired size can be found. Let $B \subseteq (I \setminus I') \times G \times (\Lambda \setminus \Lambda')$ be a generating set for $T$ with $|B| = \rank{T}$. Choose and fix some $i^* \in I \setminus I'$ and $\lambda^* \in \Lambda \setminus \Lambda'$ and set 
\[
B_I = \{ (i^*, 1_G, \lambda') : \lambda' \in \Lambda' \}, 
\quad \mbox{and} \quad 
B_\Lambda = \{ (i', 1_G, \lambda^*) : i' \in I' \}.  
\]
Then using the fact that $C_R$ is regular and that $B$ generates $T$ it is an easy exercise to verify that $B \cup B_I \cup B_\Lambda$ is a generating set for $S$ with 
\[
|B \cup B_I \cup B_\Lambda|
=
\rank{T} + |I'| + |\Lambda'|. 
\]

\vspace{2mm}

\noindent ($\geq$) Let $A$ be a generating set for $S$ with $|A| = \rank{S}$. Since $U \cup V$ is an ideal it follows that any product of elements from $A$ that belongs to $T \setminus \{ 0 \} = S \setminus (U \cup V)$ must be a product of elements from $A \cap T$. Therefore
\[
\lb A \cap T \rb = \lb A \rb \cap T = T. 
\]
Also, since $U \cap V$ is an ideal, by ($\dagger$) we conclude
\[
\lb A \setminus (U \cap V) \rb = \lb S \setminus (U \cap V) \rb = S. 
\]
Since $A \setminus (U \cap V) = A \setminus (I' \times G \times \Lambda')$ generates $S$ it must intersect every row and column of $S$ (meaning that for every $i \in I$ at least one element of $A \setminus (U \cap V)$ has $i$ as first component, and similarly for every $\lambda \in \Lambda$). We conclude
\begin{eqnarray*}
\rank{S} = |A| & \geq & 
|A \cap T| + |A \cap (I' \times G \times (\Lambda \setminus \Lambda'))| +
|A \cap ((I\setminus I') \times G \times \Lambda')| \\
& \geq & \rank{T} + |I'| + |\Lambda'|, 
\end{eqnarray*} 
as required. 
%
%
%\
%
%\
%
%\
%
%\
%
%\
\end{proof}
The steps one would have to take to find a generating set with this prescribed minimal size may be extracted from the details of the  above proof together with the proofs from \cite{Gray2005}. We shall not go into this here. 
Let us now make a few comments about how this technical looking result should be interpreted. 

\

\noindent (i) Roughly speaking one should think of $\sigma_{\mathrm{min}}$ as measuring the contribution made by the idempotents. The greater the contribution made by the idempotents, the larger the groups $H_i$ will be, and the smaller $\sigma_{\mathrm{min}}$ will need to be (since it represents the number of additional elements that are required in order to generate the whole of $G$). 

\vspace{2mm}

\noindent (ii) Features that make $S$ hard to generate include 
(a) a large zero block $C_N$
(b) $C_R$ having many connected components (i.e. $n$ large)
(c) the groups $H_i$ being small (i.e. low contribution from the idempotents). 

\vspace{2mm}

\noindent (iii) In many practical situations the formula simplifies dramatically. For example, $\sigma_{\mathrm{min}}$ is certainly bounded above by the rank of the group $G$, so if this value is small (for instance equal to $1$ or $2$) then 
the formula simplifies greatly.

\

\noindent If $S$ happens to be idempotent generated then $\Lambda' = I' = \varnothing$, $n=1$ and $H_1 = G$ and so $\sigma_{\mathrm{min}} = 0$ and we immediately obtain the following. 
\begin{corol}\label{cor_IG}
If $S = \zrms{G}{I}{\Lambda}{P}$ is idempotent generated then $\rank{S} = \max(|I|,|\Lambda|)$. 
\end{corol}
\noindent (See \cite[Theorem~2.4]{Gray2005} for a short direct proof of this fact). At the other extreme, when $S$ is an inverse completely $0$-simple semigroup, putting $P$ into Graham normal form corresponds to taking the identity matrix as structure matrix, from which we see that $\sigma_{\mathrm{min}} = \rank{G}$ and the number of connected components $n$ is given by $n=|I| = |\Lambda|$, thus in this case we obtain
\begin{corol}
If $S = \zrms{G}{I}{\Lambda}{P}$ is an inverse completely $0$-simple semigroup (that is, $S$ is a Brandt semigroup over $G$) then 
$
\rank{S} = \rank{G}  + |I| - 1. 
$
\end{corol}
This corollary is in fact an old result of Gomes and Howie; see \cite[Theorem~3.3]{key-2}. 
%
%
%
%
%
%

\begin{comment}
\begin{lem} \label{GNFones}
Let $S=\zrms{G}{I}{\Lambda}{P}$ be a completely $0$-simple semigroup. Then the matrix $P$
may normalized as $P'$ in such a way that $P'$ is in Graham normal form and every row of
$P'$ contains at least one occurrence of the identity $1_G$.
\end{lem}
\begin{proof}
Fix non-zero terms $t_{i} : i \in \Lambda$ for each row of $P$. Now multiply on the left by
${t_i}^{-1}$ for each $i \in \Lambda$. Call the new matrix $Q$. Clearly if $K$ is the subgroup
generated by the non-zero entries in $Q$ and $L$ is the subgroup generated by the non-zero
entries in $P$ then $K \leq L$. On the other hand, we know that the middle components of
$\lb E(S) \rb \cap H_{1,1}$ are all in $K$ and that $|\lb E(S) \rb \cap H_{1,1}| =
|L|$. Therefore $|K| \geq |L|$ and so $K = L$.
\end{proof}

The result above comes from a paper of R. Graham where he demonstrates the power of
viewing finite Rees matrix semigroups as labelled, directed bipartite graphs.  These
results were proven in 1968 by R. L. Graham. Later, in 1976, J. Howie carried our a very
similar analysis of $\lb E(S) \rb$ using bipartite graphs. His results are weaker than
those of Graham but the paper is shorter and, in many ways, is more accessible. In yet
another paper on the same subject Houghton (1977) proves the following:

``For any vertex $x$ of $\Gamma$, the values of $p$ on the closed paths at $x$ form a
subgroup of $G$. A different choice of base point gives a conjugate subgroup.''
\end{comment}

\section{Applications to semigroups of transformations}

As mentioned in the introduction, 
the majority of the literature on finding small generating sets of semigroups, and computing ranks, is devoted to the study of the wide variety of concrete transformation semigroups that arise in nature. Now, at first sight, the main result above may seem highly abstract, but in fact it can often be applied to give very fast simple proofs concerning ranks of concrete examples. Here we shall see a few applications of this kind to give a flavour of situations where this result can be applied to compute ranks where direct computations would be non-trivial. 

In this section we shall be concerned with transformations. Given a transformation $\alpha \in T_n$ by the \emph{rank of $\alpha$} we mean the size $|\im{\alpha}|$ of the image of $\alpha$. So the word rank will have two meanings in this section. This will not lead to any confusion since it will always be clear from context in which of the two senses the word is being used.

\subsection{A useful result for applications}

The following result, which is a straightforward application of Theorem~\ref{rankreesnonregular}, applies in many real-world examples. 

\begin{thm}
\label{thm_useful}
Let $S$ be a finite semigroup with a regular maximal $\gj$-class $J$ such that $\lb J \rb = S$. Let $H$ be a maximal subgroup of $J$, and let $i$ and $j$ be the number of $\gr$- and $\gl$-classes, respectively, of $J$. Then:
\begin{enumerate}
\item If $H$ is trivial or finite cyclic then $\rank{S} = \max(i,j)$.
\item If $\rank{H}=2$ (in particular, if $H \cong S_r$ the symmetric group with $r \geq 3$) then
\[
\rank{S} = \begin{cases}
\max(i,j)+1 & \mbox{if $i=j$ and $J$ has exactly one idempotent in} \\
& \mbox{every $\gr$- and in every $\gl$-class,} \\
\max(i,j) & \mbox{otherwise.}
\end{cases}
\]
\end{enumerate} 
\end{thm}
Of course the first case in (2) above is equivalent to saying that $J^*$ is an inverse semigroup.
The above result applies in many natural examples. Specifically it can be used to compute the ranks of the proper two-sided ideals of any of the following semigroups: the full transformation monoid $T_n$, the monoid of partial transformations $P_n$, the symmetric inverse monoid $I_n$, 
finite full linear monoids, and the
partition monoid (in the sense of \cite{East2011}), among others.

\subsection{Semigroups generated by mappings with prescribed kernels and images}\label{afots}

Many of the
examples that have been considered in the literature happen to be idempotent generated and the
question of rank can then often be answered by applying Corollary~\ref{cor_IG}.  In
order to find less trivial applications for Theorem~\ref{rankreesnonregular}, examples that are not
idempotent generated, and those that are not even connected (meaning the Graham--Houghton graphs of the Rees matrix semigroups that arise are not connected) or regular, should be considered.
In this subsection we consider a natural family of such examples.

\begin{defn}
Let $n,r \in \mathbb{N}$ with $2 < r < n$. Let $A$ be a set of $r$-subsets of $\{1,
\ldots, n\}$ and let $B$ be a set of partitions of $\{1, \ldots, n\}$, each
with $r$ equivalence classes (we shall call these \emph{partitions of weight $r$}). Define:
\[
S(A,B) = \lb \ \{ \alpha \in T_n: \im{\alpha} \in A, \ker{\alpha} \in B \}  \ \rb,
\]
the semigroup generated by all maps with image in $A$ and kernel in $B$.
\end{defn}

Clearly the semigroup $S(A,B)$ is neither regular nor idempotent generated in general.

In \cite{MR1994539}, \cite{MR1994538} and \cite{MR1933743} subsemigroups of $T_n$
generated by elements all with the same, so-called, kernel type are considered. These semigroups are
idempotent generated and in this series of papers their ranks and idempotent ranks are
computed. The semigroups $S(A,B)$ are a more general class since when $A = \{ X \subseteq
\{1, \ldots, n\} : |X| = r \}$ and $B$ is the set of kernels of a particular partition
type then we recover the examples of \cite{MR1994539}.

\begin{defn}
Let $\Gamma$ be a finite simple graph. For a subset $X$ of the vertices of $\Gamma$ define
\[
V_{0}(X) = \{ x \in X : d(x)=0   \},
\]
(where $d(x)$ denotes the degree of the vertex $x$) 
the set of all isolated vertices, and
\[
V_{+}(X) = \{ x \in X : d(x)>0   \},
\]
the vertices with non-zero degree so that $X = V_{0}(X) \cup V_{+}(X) $. Also define
$v_{0}(X) = |V_{0}(X)|$, $v_{+}(X) = |V_{+}(X)|$ and $\maxdegree{\Gamma} = \max\{d(v): v
\in V(\Gamma) \}$: the maximum degree of a vertex of the graph.
\end{defn}

Recall that the geometric edges in the Graham--Houghton graph $\Gamma(S)$ of a Rees matrix semigroup are in natural bijective correspondence with the non-zero idempotents. 
It is well known that the $\gr$-classes, and $\gl$-classes in $T_n$ may be indexed by kernels and images in a natural way, and then that  the idempotents in
$T_n$ are indexed by pairs $(\mathfrak{I},\mathfrak{K})$ where $\mathfrak{I}$ is an image and $\mathfrak{K}$ is a kernel such that $\mathfrak{I}$
is a transversal of the kernel classes of $\mathfrak{K}$ (see \cite[Chapter~2]{key-51}). This leads to the following definition.

\begin{defn} \label{theknrgraph}
Let $A$ be a set of $r$-subsets of $\{1,\ldots,n\}$ and $B$ be a set of partitions of
$\{1,\ldots,n\}$ of weight $r$. Define the bipartite graph $\Gamma(A,B)$ to have vertices
$A \cup B$ and $a \in A$ connected to $b \in B$ if and only if $a$ is a transversal of
$b$.
\end{defn}

\begin{lem} \label{genzrms}
Let $\mathcal{I} = \{ \alpha \in S(A,B): |\im{\alpha}|<r \}$ which is a two-sided ideal of $S(A,B)$. Then the Rees quotient $S(A,B)/\mathcal{I}$ is isomorphic to a (possibly non-regular) Rees matrix semigroup over the symmetric group $G \cong S_r$ and
\[
\rank{S(A,B)} = \rank{S(A,B)/\mathcal{I}}.
\]
\end{lem}
\begin{proof}
The proof is straightforward and is omitted. 
\end{proof}

Note that the graph $\Gamma(A,B)$ is just the Graham--Houghton graph of the Rees matrix semigroup $S(A,B)/\mathcal{I}$. 

\begin{thm}\label{transsgrpapplication}
Let $n,r \in \mathbb{N}$ with $2<r<n$. Let $A$ be a set of $r$-subsets of $\{1,\ldots,n\}$, 
let $B$ be a set of partitions of $\{1,\ldots,n\}$ with weight $r$, and let  
\[
S(A,B) = \lb \ \{ \alpha \in T_n: \im{\alpha} \in A, \ker{\alpha} \in B \}  \ \rb \leq T_n. 
\]
Then:
\[
\rank{S(A,B)} = \begin{cases}
\max(v_{+}(A),v_{+}(B)) + v_0(A \cup B)   & \ \mbox{if} \ \maxdegree{A \cup B} \geq 2 \\
\max(v_{+}(A),v_{+}(B)) + v_0(A \cup B) +1 & \ \mbox{if} \ \maxdegree{A \cup B} =1 \\
|A||B|r!                                     & \ \mbox{if} \ \maxdegree{A \cup B} =0
 \end{cases}
\]
where $v_{+}$, $v_0$  and $\maxdegree{A \cup B}$ refer to values of the graph $\Gamma(A,B)$.
\end{thm}
\begin{proof}
As a result of Lemma~\ref{genzrms} is is sufficient to prove the result for a
 Rees matrix semigroup $S = \zrms{S_r}{I}{\Lambda}{P}$ which is isomorphic to
$S(A,B)/\mathcal{I}$. There are three cases to consider depending on the value of
the maximum degree 
$\maxdegree{A \cup B}$ of a vertex in the graph.
We may suppose that $P$ is in Graham normal form \[
P = 
\begin{pmatrix}
C_R & 0 \\
0 & C_N
\end{pmatrix},
\]
with the notation taken from the statement of Theorem~\ref{rankreesnonregular}. 

\vspace{2mm}

\setcounter{case}{0}
\begin{case} $\maxdegree{A \cup B} = 0$.\end{case} In this case the structure matrix $P$ consists entirely of
zeros so $S(A,B)/\mathcal{I}$ is a semigroup with zero multiplication 
which has $|A||B|r!$ non-zero elements and
the result follows trivially.

\bigskip

\begin{case} $\maxdegree{A \cup B} = 1$.\end{case} In this case the matrix $C_R$ is an $|\Lambda \setminus \Lambda'| \times |I \setminus I'|$ identity matrix (a square matrix with every diagonal entry equal to $1_G$). Thus $\sigma_{\mathrm{min}}=\rank{S_r}=2$, the number of connected components $n$ of $C_R$ is $n =  |I \setminus I'| = |\Lambda \setminus \Lambda'|$, 
and so by Theorem~\ref{rankreesnonregular} we obtain 
\[
\begin{array}{llll}
\rank{S(A,B)} & = & \rank{S(A,B)/\mathcal{I}}  \\ %& \mbox{(Lemma~\ref{genzrms})} \\
              & = &
\max(|I \setminus I'|, |\Lambda \setminus \Lambda'|, \sigma_{\mathrm{min}} + n - 1) + |I'| + |\Lambda'| \\
& = & \max(n, n, 2 + n - 1) + |I'| + |\Lambda'| \\
& =& n + 1 + |I'| + |\Lambda'| 
               =  \max(v_{+}(A),v_{+}(B)) + v_0(A \cup B) +1.
\end{array}
\]

\

\begin{case} $\maxdegree{A \cup B} \geq 2$.\end{case} 
In this case the matrix $C_R$ is not a diagonal matrix and thus in particular the number of components $n$ must be strictly less than $\max(|I \setminus I'|, |\Lambda \setminus \Lambda'|)$. Also, since $\rank{G} = \rank{S_r} = 2$ it follows that $\sigma_{\mathrm{min}} \leq \rank{S_r} = 2$. We conclude that
\[
\sigma_{\mathrm{min}} + n - 1 \leq 2 + n -1 = n+1 \leq \max(|I \setminus I'|, |\Lambda \setminus \Lambda'|). 
\]
Therefore by Theorem~\ref{rankreesnonregular} we obtain 
\[
\begin{array}{llll}
\rank{S(A,B)} & = & \rank{S(A,B)/\mathcal{I}}  \\ %& \mbox{(Lemma~\ref{genzrms})} \\
              & = &
\max(|I \setminus I'|, |\Lambda \setminus \Lambda'|, \sigma_{\mathrm{min}} + n - 1) + |I'| + |\Lambda'| \\
              & = &
\max(|I \setminus I'|, |\Lambda \setminus \Lambda'|) + |I'| + |\Lambda'| 
               =  \max(v_{+}(A),v_{+}(B)) + v_0(A \cup B).
\end{array}
\]
\end{proof}

Note that the result is slightly different for $r=2$ since $S_2$ is cyclic and so has rank
$1$, not $2$.

\begin{ex}\label{ex_example}

Let $n=7$ and $r=3$ and define the set of images:
\[
A = \{ \{1,2,3\}, \{1,6,7\}, \{5,6,7\}, \{2,4,6\}, \{1,2,5\}   \}
\]
and set of partitions:
\[
B = \{ (1,4,7 |  2,5  |   3,6), (1,2,3  |4,5,6  |  7), (1,2  |  4,6,7  |  3,5)  \}.
\]
Let $S(A,B)$ be the subsemigroup of $T_n$ generated by all mappings $\alpha$ with $\im \alpha \in A$ and $\ker \alpha \in B$. 
Clearly this generating set contains $3! \times 5 \times 3 = 90$ elements. Using Theorem~\ref{transsgrpapplication} we shall now see that we can get away with far fewer generators than this. 
 The graph $\Gamma(A,B)$ is isomorphic to:
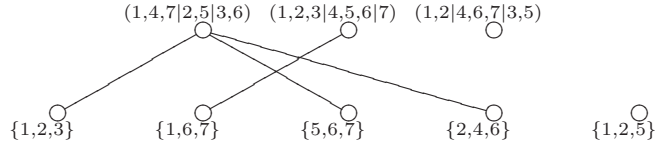
\[
\xymatrix{
& & \node \lulab{(1,4,7|2,5|3,6)} & \node \lulab{(1,2,3|4,5,6|7)} & \node
  \lulab{(1,2|4,6,7|3,5)} & \\
& \node \ldlab{\{1,2,3\}} \ar @{-} [ur] & \node \ldlab{\{1,6,7\}}  \ar @{-} [ur]  & \node \ldlab{\{5,6,7\}}  \ar @{-}
   [ul] & \node \ldlab{\{2,4,6\}}  \ar @{-} [ull] & \node \ldlab{\{1,2,5\}}  \\
& \ \ \ \ \ \ \ \ \ \ &  \ \ \ \ \ \ \ \ \ \ &  \ \ \ \ \ \ \ \ \ \ &   \ \ \ \ \ \ \ \ \
  \ &  \ \ \ \ \ \ \ \ \ \
}
\]

\vspace{-6mm}

\noindent which has two isolated vertices so that $v_0(A \cup B)=2$, $v_+(B)=2$, $v_+(A)=4$ and
maximum degree $\mathcal{MD}(A \cup B) = 3 \geq 2$. Therefore by Theorem~\ref{transsgrpapplication}:
\[
\rank{S(A,B)} = \max(2,4) + 2 = 6.
\]
\end{ex}

Example~\ref{ex_example} has been included to demonstrate the usefulness of the abstract approach offered by Theorem~\ref{rankreesnonregular}. Here we are able to conclude that the minimal number of generators for the transformation semigroup $S(A,B)$ is $6$ without having to carry out any computations with transformations at all. 
In this example, it is now an easy exercise to write down six transformations that generate the semigroup. Indeed, in general
given a set $A$ of $r$-subsets of $\{1,\ldots,n\}$ and a set of partitions $B$ of weight $r$, 
it is a routine matter to use the graph $\Gamma(A,B)$ to actually write down a minimal cardinality generating set for the semigroup $S(A,B)$. The method depends on which of the three cases of  Theorem~\ref{transsgrpapplication} we are in. 
If $\maxdegree{A \cup B} =0$, then all the mappings with image in $A$ and kernel in $B$ must be included in the generating set. 
If $\maxdegree{A \cup B} =1$, then the problem comes down to writing down a minimal generating set of a Brandt semigroup over the symmetric group $S_r$. These are nothing but the principal factors of the symmetric inverse semigroup, explicit minimal generating sets for which are given in \cite{key-2}.  
If $\maxdegree{A \cup B} \geq 2$ then a minimal generating set may be found by considering the relationship \eqref{eqn_rank} between $\rank{S}$ and $\rank{T}$ in the proof of Theorem~\ref{rankreesnonregular} combined with the argument given in the proof of \cite[Theorem~2.4]{Gray2005}. 

Of course, in general one cannot expect an efficient algorithm which takes a finite $0$-simple semigroup and computes a generating set of minimal cardinality, since any finite group is a finite simple semigroup. So the best that one could expect in the general case is an efficient reduction to a problem in computational group theory. Given an arbitrary finite Rees matrix semigroup $S$, an algorithm of this kind is as follows: 
\begin{enumerate}
\item Express $S$ as a Rees matrix semigroup in Graham--Normal form. This can be computed in polynomial time from the multiplication table of $S$. 
\item Following the notation in the statement of Theorem~\ref{rankreesnonregular}, 
compute the subgroups $H_1, \ldots, H_n$ of $G$. 
\item Then one needs a group theoretic algorithm taking $H_1, \ldots, H_n$ and $G$ as input, and outputting (i) a list $g_1, \ldots, g_n$ of elements of $G$ such that 
\[
\rank{
G:\bigcup_{i=1}^n g_i H_i g_i^{-1} \; | \; g_1, \ldots, g_n \in G
}
\]
is as small as possible over all possible choices of $g_1, \ldots, g_n$ (i.e. is equal to $\sigma_{\mathrm{min}}$), and (ii) a set $x_1, \ldots, x_{\sigma_{\mathrm{min}}}$ of elements of $G$ such that the set
$
\left( \bigcup_{i=1}^n g_i H_i g_i^{-1} \right) \cup \{ x_1, \ldots, x_{\sigma_{\mathrm{min}}} \} 
$
generates $G$. 
\item Once the Graham--Normal form representation is given, along with the list of elements $g_1, \ldots, g_n, x_1, \ldots,  x_{\sigma_{\mathrm{min}}}$, a combination of the arguments given in the proof of Theorem~\ref{rankreesnonregular} above, together with the proof given of the forward implication of \cite[Theorem~7.1]{Gray2005}, can be used to write down an explicit minimal generating set for the given finite Rees matrix semigroup $S$.   
\end{enumerate}
In general, the most time consuming part of the above algorithm will be step (3). Indeed, the reason that in the spacial case of the semigroups $S(A,B)$ minimal generating sets can easily be computed, is due to the fact that step (3) is a triviality in such examples.

%%%
%%%
%%%

It is well known that the general linear group over a finite field has a generating set consisting of two elements. Thus there will be a  
natural analogue of Theorem~\ref{transsgrpapplication}
for subsemigroups of finite full linear monoids generated by matrices with prescribed column and row spaces. 
\subsection{Generating transformation monoids}

In this subsection we are interested in the following very general situation. Suppose we are given a set $A$ of transformations from $T_n$ and we are interested in the semigroup $S$ generated by $A$. Now, the generating set $A$ we have been given may not be at all efficient, and it would be of interest, if possible, to replace $A$ by a smaller set $B$ that also generates $S$. So, we would like to know $\rank{\lb A \rb}$ given $A$. In particular we would like to investigate worst case scenarios, in other words, given that $S$ is a transformation semigroup on $n$ points, how many generators (in the worst case) will we need to generate $S$? 
That is, what is $\max\{ \rank{S} : S \leq T_n \}$? As far as the author is aware, this natural question does not seem to have been considered anywhere in the literature. Some preliminary results will be given here, but there are still many open problems in this area that may well provide an interesting new direction to explore for those interested in transformation semigroups. 

This is a classical question for the symmetric group $S_n$ in the case of permutation groups. Indeed, a well-known result due to Jerrum \cite{Jerrum1982} says that any subgroup of $S_n$ can be generated by at most $n-1$ elements. Jerrum's result can be used to compute a base and strong generating set for an arbitrary subgroup of $S_n$ in polynomial time. The bound $n-1$ is not best possible. McIver and Neumann \cite{McIverNeumann1987} (see also \cite{MR1028457} and \cite[Volume 1, Section 8]{Handbook}) showed that any subgroup of $S_n$ can be generated by at most $\lfloor n/2 \rfloor$ if $n>3$. 
This result makes use of the classification of finite simple groups, and is best possible for $n\geq 3$ (consider the group generated by $\lfloor n/2 \rfloor$ disjoint transpositions). 
\begin{thm}\cite{McIverNeumann1987} \label{cameron}
$\rank{G} \leq \max(2, \lfloor n/2 \rfloor)$ for all $G \leq S_n$.
\end{thm}
A similar result for subsemigroups of the full transformation semigroup would be of
interest. In general this still seems like a difficult problem:

\begin{openproblem}
Determine a formula for $\max\{ \rank{S} : S \leq T_n \}$.
\end{openproblem}

If we add a number of hypotheses we are able to 
apply Theorem~\ref{rankreesnonregular} to
obtain a positive
result of the the above type.
%\
%
%
%\
\begin{thm}\label{thm_NeumTn}
Let $n \geq 4$ and let $1 < r < n$. Every regular subsemigroup of $T_n$ that is generated
by mappings all with rank $r$, and has a unique maximal $\gj$-class, is generated by at
most $S(n,r)$ elements, where $S(n,r)$ denotes the Stirling number of the second kind. 
\end{thm}
\begin{proof}
Let $S$ be a regular subsemigroup of $T_n$ generated by mappings of rank $r$ and with a
unique maximal $\gj$-class. 
Recall that the $\gd$-class $D_r$ of $T_n$ of all elements of rank $r$ has ${n \choose r}$ $\gl$-classes and $S(n,r)$ $\gr$-classes. Let $J_M$ be the unique maximal $\gj$-class of $S$. 
Since $S$ is a regular subsemigroup of $T_n$ generated by maps of rank $r$ and with unique maximal $\gj$-class $J_M$, it follows that 
\[
J_M = \{ \alpha \in S : |\im \alpha| = r \}. 
\]
Then
$\rank{S} = \rank{{J_M}^*}$ where the principal factor ${J_M}^*$ is isomorphic to a
completely $0$-simple $\zrms{G}{I}{\Lambda}{P}$ where $G \leq S_r$ and the matrix $P$ has
at most ${n \choose r}$ connected components. By Theorem~\ref{cameron}, since $G \leq
S_r$, it follows that $\rank{G} \leq \max(2,\lfloor r/2 \rfloor)$. Then by Theorem~\ref{rankreesnonregular}
\[
\rank{S} \leq \max(S(n,r), { n \choose r }, \max(2,\lfloor r/2 \rfloor)+  {n \choose r} - 1).
\]
Since $S(n,r) > {n \choose r}$ for $1 < r < n$ we obtain
\[
\max(S(n,r), { n \choose r }, \max(2,\lfloor r/2 \rfloor) +  {n \choose r} - 1) = \max(S(n,r),
\max(2,\lfloor r/2 \rfloor) +  {n \choose r} - 1).
\]
Also, for $n \geq 4$ and $1<r<n$ we have:
\[
S(n,r) > \max(2,\lfloor r/2 \rfloor) +  {n \choose r} - 1.
\]
We conclude that $\rank{S} \leq S(n,r)$.
%%
 %
\begin{comment}
This is because for every $r$-subset of $X_n$ we can associate the ``clock'' partition as
described in Definition~\ref{clockpartition}. In addition to these ${n \choose r}$
partitions of weight $r$ there also exists, for each set $s \in \{ \{ 1,3 \}, \{ 1,4 \},
\ldots, \{1, n-2 \} \}$, a partition with $s$ as one of the classes.  There are at least
$n-3$ of these and $n-3 > \lfloor r/2 \rfloor -1$ for $n \geq 4$ and $1<r<n$. It follows
that:
\[
S(n,r) \geq (n-3) + {n \choose r} > \lfloor r/2 \rfloor + {n \choose r} - 1
\]
as required.
\end{comment}
 %
%%
\end{proof}
This result is best possible in the sense that for $1 < r < n$ this upper bound is attained by the subsemigroup 
\[
K(n,r) = \{ \alpha \in T_n : |\im \alpha| \leq r \},
\]
which satisfies all of the hypotheses of the theorem, and satisfies
$\rank{K(n,r)} = S(n,r)$; see \cite{key-13}. 
It is natural to ask whether one can remove the hypothesis that the subsemigroup must have a unique maximal $\gj$-class. 
%\
%\
\begin{openproblem}
Let $n \geq 4$ and $1 < r < n$. Is it true that any regular subsemigroup of $T_n$ that is
generated by mappings of rank $r$ is generated by at most $S(n,r)$ elements?
\end{openproblem}

\subsection{Inverse semigroups of transformations}

For inverse semigroups, the analogue of Cayley's theorem is the Wagner--Preston Theorem which says that any (finite) inverse semigroup can be embedded in a (finite) symmetric inverse semigroup. Here we see that the  McIver--Neumann admits a generalisation to subsemigroups of the symmetric inverse semigroup generated by maps of fixed rank (the McIver--Neumann result is the special case when $r=n$).

\begin{thm}
Let $S$ be a subsemigroup of $I_n$ generated by a set of maps all of rank $r$ for some $1 < r \leq n$. If $S$ is an inverse semigroup then
\[
\rank{S} \leq {n \choose r} \max\left(2,\left\lfloor \frac{r}{2} \right\rfloor\right). 
\]
%
%
%\
%
\end{thm}
\begin{proof}
Let $J_1, \ldots, J_k$ be the maximal $\gj$-classes of $S$. As in the proof of Theorem~\ref{thm_NeumTn} above, since $S$ is a regular subsemigroup and is generated by maps of rank $r$ it follows that
\[
\bigcup_{i=1}^k J_i = \{ \alpha \in S : |\im \alpha| = r\}.
\] 
For $i=1,\ldots, k$, let $n_i$ be the number of $\gr$-classes in $J_i$, which also equals the number of $\gl$-classes in $J_i$ and the number of idempotents in $J_i$ (since $S$ is inverse). Since $S$ is generated by a set of maps of rank $r$ it now follows that
\[
\rank{S} = \sum_{i=1}^k \rank{J_i^*}
\]
where each $J_i^*$ is a finite completely $0$-simple inverse semigroup (i.e. is a Brandt semigroup). Let $G_i$ denote the maximal subgroup of this Rees matrix semigroup $J_i^*$. Since each $G_i$ is isomorphic to a subgroup of $S_r$ it follows from Theorem~\ref{cameron} that $\rank{G_i} \leq \max\left(2,\left\lfloor \frac{r}{2} \right\rfloor\right)$ for all $i=1,\ldots,k$. Since there are exactly ${n \choose r}$ idempotents $\epsilon$ in $I_n$ with $| \im \epsilon | = r$, it follows that $n_1 + \ldots + n_k \leq {n \choose r}$ and of course also that $k \leq {n \choose r}$. Combining these observations with the rank formula from Theorem~\ref{rankreesnonregular} then gives:
\begin{eqnarray*}
\rank{S} & = & \sum_{i=1}^k \rank{J_i^*} = 
\sum_{i=1}^k \max(n_i, \rank{G_i} + n_i -1) \\
& = & \sum_{i=1}^k (\rank{G_i} + n_i -1) 
 \leq  \sum_{i=1}^k (\max\left(2,\left\lfloor \frac{r}{2} \right\rfloor\right) + n_i -1) \\
& = & k(\max\left(2,\left\lfloor \frac{r}{2} \right\rfloor\right) -1) + (n_1 + \ldots + n_k) \\
& \leq & {n \choose r}(\max\left(2,\left\lfloor \frac{r}{2} \right\rfloor\right) -1) + {n \choose r} 
 =  {n \choose r}\max\left(2,\left\lfloor \frac{r}{2} \right\rfloor\right),
\end{eqnarray*}
as required.
\end{proof}
This result is best possible for $2 < r \leq n$. Indeed, given $r,n \in \mathbb{N}$ with $3 < r \leq n$ the bound is attained by 
taking, for each of the ${n \choose r}$ distinct copies of $S_r$ in the $\gd$-class $D_r$ of $I_n$ (of all maps of rank $r$) a set of $\lfloor \frac{r}{2} \rfloor$ disjoint transpositions, and setting $S$ to be the subsemigroup generated by these ${n \choose r}\lfloor \frac{r}{2} \rfloor$ elements (in the special case $r=3$ we take all copies of $S_3$ as our inverse generating set to attain the bound). The following natural problem remains open.

\begin{openproblem}
Find $\max \{ \rank{S} : S \leq I_n \}$ where $S$ is an inverse subsemigroup of $I_n$.
\end{openproblem}

A related alternative line of investigation for transformation semigroups is that of random generation. There are numerous interesting results about random generation of finite groups. In 1969, John Dixon \cite{Dixon69} proved that the probability that two random permutations in the symmetric group $S_n$ generate $S_n$ or $A_n$ tends to $1$ as $n \rightarrow \infty$. This was extended in \cite{KL1990} and \cite{LS1995} to arbitrary sequences of non-abelian finite simple groups.
In a quite amazing recent paper in this area \cite{ZP2011} Jaikin-Zapirain and  Pyber find an explicit formula  for the number of random elements needed to generate a finite $d$-generated group $G$ with high probability. As a corollary they prove that if $G$ is a $d$-generated linear group of dimension $n$ then $cd + \mathrm{log} \ n$ random generators suffice, for some absolute constant $c$.

Similar questions may be asked for finite semigroups. Cameron \cite{Cameron2013} recently began such an investigation for the full transformation monoid $T_n$. As he points out, one has to be careful to ask the right questions here, and one should not expect the obvious analogue of Dixon's result to hold true. Indeed, $T_n$ requires three generators, and any generating set of minimal size for $T_n$ has the form $\{\alpha, \beta, \gamma\}$ where $\langle \alpha, \beta \rangle = S_n$ and $|\im \gamma| = n-1$. In general, if a monoid $M$ is generated by a set of transformations $A$, then the group of permutations in $M$ is generated by the permutations in $A$. Since permutations are exponentially scarce in $T_n$, we have to choose a huge number of random elements in order to generate $T_n$ with high probability. Cameron suggest a different approach, leading him to conjecture that the probability that two random elements in $T_n$ generate a synchronising monoid (one that contains a constant mapping) tends to $1$ as $n \rightarrow \infty$.

\section*{Acknowledgements}

Some of the results given here appear first appeared in the PhD thesis of the author \cite{GrayThesis}.

%This research was carried out while the author was a research student at the University of St Andrews. This paper is drawn from results in [9, Chap- ter 3]. The author expresses his gratitude to Professor Nikola Ruÿskuc for his helpful comments during the preparation of this paper.
%Some of the results presented here appeared in the thesis of the author (cite my thesis) supported by EPSRC??  

%\bibliography{MasterBibliography}
%\bibliographystyle{alphaabbrv}

\def\cprime{$'$}

\end{document}